\documentclass[10pt]{amsart}

\usepackage[left=2.8cm, right=2.8cm, top=2.2cm, bottom=2.2cm]{geometry}

\usepackage{amsmath,amssymb,amsthm,mathrsfs,stmaryrd,latexsym}
\usepackage{amsfonts}

\usepackage{tcolorbox}
\tcbuselibrary{skins, breakable, theorems}

\usepackage{enumerate} 
\usepackage{tikz}
\usepackage{etoolbox}

\def\Luoma#1{\uppercase\expandafter{\romannumeral#1}}
\def\luoma#1{\romannumeral#1}

\usepackage{url}

\newtheorem{mythm}[subsection]{Theorem}
\newtheorem{mylem}[subsection]{Lemma}
\newtheorem{myprop}[subsection]{Proposition}
\newtheorem{mycor}[subsection]{Corollary}

\theoremstyle{definition}
\newtheorem{mydefn}[subsection]{Definition}

\theoremstyle{remark}
\newtheorem{myrem}[subsection]{Remark}

\numberwithin{equation}{subsection}

\usepackage[all]{xy}
\usepackage{graphicx}
\usepackage{subfigure}

\newcommand{\bb}{\mathbb}
\newcommand{\ca}{\mathcal}
\newcommand{\ak}{\mathfrak}
\newcommand{\scr}{\mathscr}

\newcommand{\mrm}{\mathrm}

\newcommand{\op}{\operatorname}

\newcommand{\ho}{\op{Hom}}
\newcommand{\ke}{\op{Ker}}
\newcommand{\im}{\op{Im}}

\newcommand{\df}{\mrm{d}}

\newcommand{\iso}{\stackrel{\sim}{\longrightarrow}}

\newcommand{\plim}{\varprojlim}
\newcommand{\colim}{\mathop{\mathrm{colim}}}

\def\omhat#1#2{\widehat{\Omega}_{\ca{O}_{#1}}^1(\ca{O}_{#2})}

\title[Faltings Extension and Hodge-Tate Filtration for Abelian Varieties]{Faltings Extension and Hodge-Tate Filtration for Abelian Varieties over $p$-adic Local Fields with Imperfect Residue Fields}
\author{Tongmu He}
\address{Tongmu He, Institut des Hautes \'Etudes Scientifiques, 35 route de Chartres, 91440 Bures-sur-Yvette, France}
\email{hetm15@ihes.fr}

\usepackage{hyperref}
\hypersetup{colorlinks,
citecolor=black,
filecolor=black,
linkcolor=black,
urlcolor=black,
pdftitle={Faltings Extension and Hodge-Tate Filtration for Abelian Varieties over $p$-adic Local Fields with Imperfect Residue Fields},
pdfauthor={Tongmu He},
pdftex}

\begin{document}
\maketitle
\begin{abstract}
	Let $K$ be a complete discrete valuation field of characteristic $0$, with not necessarily perfect residue field of characteristic $p>0$. We define a Faltings extension of $\ca{O}_K$ over $\bb{Z}_p$, and we construct a Hodge-Tate filtration for abelian varieties over $K$ by generalizing Fontaine's construction \cite{fontaine1982formes} where he treated the perfect residue field case.
	
	\noindent
	\emph{2010 Mathematics Subject Classification.} Primary 14F30; Secondary 14G20, 14K15.
\end{abstract}

\footnotetext{This draft is going to appear in \emph{Canadian Mathematical Bulletin}. See \url{https://doi.org/10.4153/S0008439520000399}.}

\tableofcontents

\section{Introduction}
\subsection{}
Let $K$ be a complete discrete valuation field of characteristic $0$, with residue field $k$ of characteristic $p>0$. Let $\overline{K}$ be an algebraic closure of $K$, $G_K$ the Galois group of $\overline{K}$ over $K$, $C$ the $p$-adic completion of $\overline{K}$. We denote by $C(r)$ the $r$-th Tate twist. For an abelian variety $X$ over $K$, we denote its Tate module by $T_p(X)$. When $k$ is perfect and $X$ has good reduction, Tate \cite{tate1967p} constructed a canonical $G_K$-equivariant exact sequence
\begin{align}\label{eq:HTdecom-abvar}
\xymatrix{
	0\ar[r]& H^1(X,\ca{O}_X)\otimes_KC(1)\ar[r]&\ho_{\bb{Z}_p}(T_p(X),C(1))\ar[r]& H^0(X,\Omega_{X/K}^1)\otimes_KC\ar[r]&0.
}
\end{align}
In the same paper, Tate also computed the Galois cohomology groups of $C(r)$. He proved in particular that $H^1(G_K,C(r))=0$ for any $r\neq 0$, which implies that the sequence \eqref{eq:HTdecom-abvar} has a $G_K$-equivariant splitting, and that $H^0(G_K,C(r))=0$ for any $r\neq 0$, which implies that the splitting is unique. Tate conjectured that for any proper smooth scheme $X$ over $K$, there is a canonical $G_K$-equivariant decomposition (called the \emph{Hodge-Tate decomposition})
\begin{align}\label{eq:HTdecom}
H^n_{\textrm{\'et}}(X_{\overline{K}},\bb{Q}_p)\otimes_{\bb{Q}_p}C=\bigoplus\nolimits_{i=0}^n H^{i}(X,\Omega_{X/K}^{n-i})\otimes_K C(i-n).
\end{align}
Then subsequently Raynaud used the semistable reduction theorem to show that any abelian variety over $K$ admits a Hodge-Tate decomposition (\cite{sga7-1} \Luoma{9} 3.6, 5.6). Afterwards, Fontaine \cite{fontaine1982formes} gave a new proof for general abelian varieties. He constructed a canonical map $H^0(X,\Omega_{X/K}^1)\to \ho_{\bb{Z}_p[G_K]}(T_p(X),C(1))$, by computing $\Omega_{\ca{O}_{\overline{K}}/\ca{O}_K}^1$ and pulling back differentials. The conjecture of Tate was finally settled by Faltings \cite{faltings1988p,faltings2002almost} and Tsuji \cite{tsuji1999p,tsuji2002semi} independently.

When $k$ is not necessarily perfect, Hyodo proved that there is still an exact sequence \eqref{eq:HTdecom-abvar} for abelian varieties with good reduction, following the same argument as in \cite{tate1967p} (\cite{hyodo1986hodge} Remark 1). But the sequence does not split in general (\cite{hyodo1986hodge} Theorem 3). In this paper, we will construct the exact sequence \eqref{eq:HTdecom-abvar} for general abelian varieties by generalizing Fontaine's method to the imperfect residue field case. 

We remark that Scholze \cite{scholze2013perfectoid} has generalized the conjecture of Tate to any proper smooth rigid-analytic variety $X$ over $C$. He proved that there is a canonical filtration (called the \emph{Hodge-Tate filtration}) $\op{Fil}^\bullet$ on $H^{n}_{\textrm{\'et}}(X,\bb{Q}_p)\otimes_{\bb{Q}_p}C$, such that
\begin{align}\label{eq:HTfiltration}
\op{Fil}^i(H^{n}_{\textrm{\'et}}(X,\bb{Q}_p)\otimes_{\bb{Q}_p}C)/\op{Fil}^{i+1}(H^{n}_{\textrm{\'et}}(X,\bb{Q}_p)\otimes_{\bb{Q}_p}C)=H^i(X,\Omega_{X/C}^{n-i})\otimes_CC(i-n).
\end{align}

\subsection{}
For any abelian group $M$, we set
\begin{align}
T_p(M)=\ho_{\bb{Z}}(\bb{Z}[1/p]/\bb{Z},M),\ V_p(M)=\ho_{\bb{Z}}(\bb{Z}[1/p],M).
\end{align}

In section \ref{sec:falext}, we construct a Faltings extension of $\ca{O}_K$ over $\bb{Z}_p$. It is a canonical exact sequence of $C$-$G_K$-modules which splits as a sequence of $C$-modules (cf. \ref{thm:falext}),
\begin{align}\label{eq:intro-falext}
\xymatrix{
	0\ar[r]&C(1)\ar[r]^-{\iota}& V_p(\Omega_{\ca{O}_{\overline{K}}/\ca{O}_K}^1)\ar[r]^-{\nu}&C\otimes_{\ca{O}_C}(\ca{O}_{\overline{K}}\otimes_{\ca{O}_K} \Omega_{\ca{O}_{K}/\bb{Z}_p}^1)^\wedge\ar[r]&0,
}
\end{align}
where $(-)^\wedge$ denotes the $p$-adic completion.
Based on Hyodo's computation of Galois cohomology (cf. \ref{thm:hyodo-galcoh}), we will show that the connecting map of the above sequence
\begin{align}\label{eq:intro-connmap}
\delta: (C\otimes_{\ca{O}_C}(\ca{O}_{\overline{K}}\otimes_{\ca{O}_K} \Omega_{\ca{O}_{K}/\bb{Z}_p}^1)^\wedge)^{G_K}\longrightarrow H^1(G_K,C(1))
\end{align}
is an isomorphism (cf. \ref{cor:falext-conn}).

Following Fontaine, we deduce from the above sequence and its cohomological properties a canonical $K$-linear injective homomorphism (cf. \ref{thm:fon-inj})
\begin{align}\label{eq:intro-foninj}
\rho:H^0(X,\Omega_{X/K}^1)\longrightarrow \ho_{\bb{Z}_p[G_K]}(T_p(X),V_p(\Omega_{\ca{O}_{\overline{K}}/\ca{O}_K}^1)).
\end{align}
The arguments are essentially the same as in \cite{fontaine1982formes}.

Our main result can be stated as follows (cf. \ref{thm:HTfil}, \ref{ssec:7.5}, \ref{ssec:7.6}):
\begin{mythm}\label{thm:main}
	For any abelian variety $X$ over $K$, there is a canonical exact sequence of $C$-$G_K$-modules
	\begin{align}\label{eq:intro-HTfil}
	\xymatrix{
		0\ar[r]& H^1(X,\ca{O}_X)\otimes_KC(1)\ar[r]^-{\psi}&\ho_{\bb{Z}_p}(T_p(X),C(1))\ar[r]^-{\phi}& H^0(X,\Omega_{X/K}^1)\otimes_KC\ar[r]&0
	}
	\end{align}
	satisfying the following properties:
	\begin{enumerate}[\rm (i)]
		\item Any $C$-linear retraction of $\iota$ in \eqref{eq:intro-falext} induces a $C$-linear section of $\phi$. More precisely, we have a commutative diagram
		\begin{align}\label{diam:intro-HTcomp1}
		\xymatrix{
			\ho_{\bb{Z}_p}(T_p(X),C(1))\ar[r]^-{\phi}&H^0(X,\Omega_{X/K}^1)\otimes_K C\ar[d]^-{\rho}\\
			&\ho_{\bb{Z}_p}(T_p(X),V_p(\Omega_{\ca{O}_{\overline{K}}/\ca{O}_K}^1))\ar[lu]^-{\pi}
		}
		\end{align}
		where $\rho$ is induced by the map \eqref{eq:intro-foninj} and $\pi$ is induced by any retraction of $\iota$ in \eqref{eq:intro-falext}.
		\item The connecting map $\delta'$ associated to \eqref{eq:intro-HTfil} fits into a commutative diagram
		\begin{align}\label{diam:intro-HTcomp2}
		\xymatrix{
			H^0(X,\Omega_{X/K}^1)\ar[d]_-{\rho}\ar[r]^-{\delta'}&H^1(G_K,H^1(X,\ca{O}_X)\otimes_KC(1))\ar[d]\\
			\ho_{\bb{Z}_p}(T_p(X),V_p(\Omega_{\ca{O}_{\overline{K}}/\ca{O}_K}^1))\ar[r]^-{\pi'}&\ho_{\bb{Z}_p}(T_p(X),C\otimes_{\ca{O}_C}(\ca{O}_{\overline{K}}\otimes_{\ca{O}_K} \Omega_{\ca{O}_{K}/\bb{Z}_p}^1)^\wedge)
		}
		\end{align}
		where $\rho$ is the map \eqref{eq:intro-foninj}, $\pi'$ is induced by $-\nu$ of \eqref{eq:intro-falext}, and the unlabelled arrow is induced by $\delta^{-1}$ \eqref{eq:intro-connmap} and $\psi$ of \eqref{eq:intro-HTfil}.
	\end{enumerate}
\end{mythm}

\begin{mycor}\label{cor:main}
	For any abelian variety $X$ over $K$,
	the sequence \eqref{eq:intro-HTfil} splits if and only if the image of $\rho$ \eqref{eq:intro-foninj} lies in $\ho_{\bb{Z}_p[G_K]}(T_p(X),C(1))$. In fact, when it splits, the splitting is unique.
\end{mycor}

\begin{myrem}\label{rem:HTss}
	Caraiani and Scholze \cite{caraiani2017generic} have constructed a relative version of Hodge-Tate filtration for proper smooth morphisms of adic spaces. And recently, Abbes and Gros \cite{abbes2020suite} have constructed a relative version of Hodge-Tate spectral sequence for projective smooth morphisms of logarithmic schemes. Unlike these works that rely on advanced theories and results, our proof for abelian varieties uses only basic algebraic geometry and $p$-adic Galois cohomology computation of Tate and Hyodo. For instance, we do not use Faltings' almost purity theorem.
\end{myrem}

\subsection{Acknowledgement}
I would like to express my sincere gratitude to my advisor Ahmed Abbes. He led me to this topic and spent a lot of time on discussing with me. He gave me plenty of helpful advice on both studying and writing. I would like to thank the referee for improving some statements and many other helpful comments. This work was done while I was a master's student at D\'epartement de Math\'ematiques d'Orsay, Universit\'e Paris-Saclay, supported by Fondation Math\'ematique Jacques Hadamard (FMJH).

\section{Notation}
\subsection{}
Let $K$ be a complete discrete valuation field of characteristic $0$, with residue field $k$ of characteristic $p>0$. Let $\overline{K}$ be an algebraic closure of $K$, $G_K$ the Galois group of $\overline{K}$ over $K$. Let $C$ be the $p$-adic completion of $\overline{K}$, $v_p$ the valuation on $C$ such that $v_p(p)=1$, $|\ |_p$ the absolute value on $C$ such that $|p|_p=1/p$.  We fix a complete discrete valuation subfield $K_0$ of $K$ such that $\ca{O}_{K_0}/p\ca{O}_{K_0}=k$ (by Cohen structure theorem, cf. \cite{ega4-1} $0_{\op{\Luoma{4}}}$ 19.8.6). We remark that $K/K_0$ is a totally ramified finite extension. We fix elements $(u_i)_{i\in I}$ of $\ca{O}_{K_0}$ such that the reductions $(\overline{u_i})_{i\in I}$ form a $p$-base of $k$. For each $i\in I$, we fix elements $(w_{im})_{m\geq 0}$ of $\ca{O}_{\overline{K}}$ such that $w_{i,m+1}^p=w_{i,m}$ and $w_{i,0}=u_i$. We denote by $(e_i)_{i\in I}$ the standard basis of $\oplus_{i\in I} \bb{Z}$.

\subsection{}
For any discrete valuation field $L$ of characteristic $0$, with residue field of characteristic $p$, we denote by
\begin{align}\label{eq:OmegaK1}
\widehat{\Omega}_{\ca{O}_L}^1=(\Omega_{\ca{O}_{L}/\bb{Z}_p}^1)^\wedge
\end{align}
the $p$-adic completion of the module of differentials of $\ca{O}_L$ over $\bb{Z}_p$.

For any algebraic extension $L'$ over $L$, we set
\begin{align}\label{eq:OmegaK2}
\omhat{L}{L'}=\colim_{L_1/L}\  \widehat{\Omega}_{\ca{O}_{L_1}}^1,
\end{align} 
where $L_1$ runs through all finite subextensions of $L'/L$. We remark that $\omhat{L}{L'}=\widehat{\Omega}_{\ca{O}_{L_1}}^1(\ca{O}_{L'})$ for any finite subextension $L_1$ of $L'/L$, and that $\omhat{L}{L}=\widehat{\Omega}_{\ca{O}_L}^1$.

\subsection{}\label{ssec:2.3}
For any abelian group $M$, we define
\begin{align}
T_p(M)&=\plim_{x\mapsto px} M[p^n]=\ho_{\bb{Z}}(\bb{Z}[1/p]/\bb{Z},M),\\
V_p(M)&=\plim_{x\mapsto px} M=\ho_{\bb{Z}}(\bb{Z}[1/p],M).
\end{align}
Being an inverse limit of $\bb{Z}$-modules each killed by some power of $p$, $T_p(M)$ is a $p$-adically complete $\bb{Z}_p$-module (\cite{jannsen1988cont} 4.4). If $M$ is $p$-primary torsion, then $V_p(M)=T_p(M)\otimes_{\bb{Z}_p}\bb{Q}_p$, and thus it has a natural $\bb{Q}_p$-module structure. If $M$ is a $\bb{Z}_p$-module, then $T_p(M)=\ho_{\bb{Z}_p}(\bb{Q}_p/\bb{Z}_p,M)$, $V_p(M)=\ho_{\bb{Z}_p}(\bb{Q}_p,M)$. We set $\bb{Z}_p(1)=T_p(\ca{O}_{\overline{K}}^\times)$, a free $\bb{Z}_p$-module of rank $1$ with continuous $G_K$-action. For any $\bb{Z}_p$-module $M$ and $r\in \bb{Z}$, we set $M(r)=M\otimes_{\bb{Z}_p}\bb{Z}_p(1)^{\otimes r}$, the $r$-th Tate twist of $M$.
Let $X$ be an abelian variety over $K$. We set $T_p(X)=T_p(X(\overline{K}))$ and $V_p(X)=V_p(X(\overline{K}))$.

\section{Review of Hyodo's Computation of Galois Cohomology Groups of $C(r)$}
\begin{mylem}\label{lem:dvr-diffinj}
	Let $B/A$ be a finite extension of discrete valuation rings, whose fraction field extension and residue field extension are both separable. We assume that $A$ is henselian, or that $B/A$ is totally ramified. Let $R$ be a subring of $A$. Then the canonical map $B\otimes_A \Omega_{A/R}^1\to \Omega_{B/R}^1$ is injective.
\end{mylem}
\begin{proof}
	After replacing $A$ by its maximal unramified extension in $B$, we may assume that $B$ is totally ramified over $A$. Hence, $B$ is of the form $A[X]/(f(X))$ for some irreducible polynomial $f\in A[X]$. Let $x$ be the image of $X$ in $B$. Then we have
	\begin{align}
	\Omega_{B/R}^1=(B\otimes_A\Omega_{A/R}^1\oplus B\df X)/B(\df_Af(x)+f'(x)\df X),
	\end{align}
	where $\df_A f\in A[X]\otimes_A \Omega_{A/R}^1$ is obtained by differentiating the coefficients of $f$.
	Since $f'(x)\neq 0$, the canonical map $B\otimes_A\Omega_{A/R}^1\to \Omega_{B/R}^1$ is injective.
\end{proof}

\begin{mylem}[\cite{hyodo1986hodge} 4-4]\label{lem:hyodo-K0diff}
	There is an isomorphism of $\ca{O}_{K_0}$-modules
	\begin{align}\label{eq:hyodo-K0diff}
	(\oplus_{i\in I}\ca{O}_{K_0})^\wedge\iso \widehat{\Omega}_{\ca{O}_{K_0}}^1 ,\ e_i\mapsto \df\log u_i,\ \forall i\in I.
	\end{align}
\end{mylem}
\begin{proof}
	As $(\overline{u_i})_{i\in I}$ form a $p$-base of the residue field of $\ca{O}_{K_0}$, we have $\Omega_{\ca{O}_{K_0}/\bb{Z}_p}^1\otimes_{\bb{Z}_p}\bb{Z}_p/p\bb{Z}_p=\Omega_{k/\bb{F}_p}^1= \oplus_{i\in I}k$, where $e_i$ corresponds to $\df\log \overline{u_i}$. Since $\ca{O}_{K_0}$ is flat over $\bb{Z}_p$ and $k$ is formally smooth over $\bb{F}_p$, $\ca{O}_{K_0}/p^n\ca{O}_{K_0}$ is formally smooth over $\bb{Z}_p/p^n\bb{Z}_p$ for each $n\geq 1$ (\cite{ega4-1} $0_{\op{\Luoma{4}}}$ 19.7.1, \cite{stacks-project} \href{https://stacks.math.columbia.edu/tag/031L}{031L}). In particular, $\Omega_{\ca{O}_{K_0}/\bb{Z}_p}^1\otimes\bb{Z}_p/p^n\bb{Z}_p$ is a projective $\ca{O}_{K_0}/p^n\ca{O}_{K_0}$-module. Hence, we have an exact sequence
	\begin{align}
	\xymatrix{
	0\ar[r]& \Omega_{\ca{O}_{K_0}/\bb{Z}_p}^1\otimes\bb{Z}_p/p\bb{Z}_p\ar[r]^-{\cdot p^{n-1}}\ar[r]&\Omega_{\ca{O}_{K_0}/\bb{Z}_p}^1\otimes\bb{Z}_p/p^n\bb{Z}_p\ar[r]&\Omega_{\ca{O}_{K_0}/\bb{Z}_p}^1\otimes\bb{Z}_p/p^{n-1}\bb{Z}_p\ar[r]&0,
	}
	\end{align}
	from which we get isomorphisms $\oplus_{i\in I}\ca{O}_{K_0}/p^n\ca{O}_{K_0}\iso\Omega_{\ca{O}_{K_0}/\bb{Z}_p}^1\otimes\bb{Z}_p/p^n\bb{Z}_p$ by induction. The conclusion follows by taking limit over $n$.
\end{proof}

\begin{myprop}[\cite{hyodo1986hodge} 4-2-1]\label{prop:hyodo-K/Z-diff}
	There is an exact sequence of $\ca{O}_K$-modules
	\begin{align}\label{eq:hyodo-K/Z-diff}
	\xymatrix{
		0\ar[r]&(\oplus_{i\in I}\ca{O}_K)^\wedge\ar[r]^-{\theta}&\widehat{\Omega}_{\ca{O}_K}^1\ar[r]& \Omega_{\ca{O}_K/\ca{O}_{K_0}}^1\ar[r]&0,
	}
	\end{align}
	where $\theta(e_i)=\df\log u_i$ for any $i\in I$.
\end{myprop}
\begin{proof}
	The sequence of modules of differentials of $\ca{O}_K/\ca{O}_{K_0}/\bb{Z}_p$,
	\begin{align}\label{eq:hyodo-K/Z-diff3}
	\xymatrix{
		0\ar[r]&\ca{O}_K\otimes_{\ca{O}_{K_0}}\Omega_{\ca{O}_{K_0}/\bb{Z}_p}^1\ar[r]&\Omega_{\ca{O}_{K}/\bb{Z}_p}^1\ar[r]& \Omega_{\ca{O}_K/\ca{O}_{K_0}}^1\ar[r]&0,
	}
	\end{align}
	is exact by \ref{lem:dvr-diffinj}. Passing to $p$-adic completions, as $\Omega_{\ca{O}_K/\ca{O}_{K_0}}^1$ is killed by a power of $p$, we still get an exact sequence (\cite{stacks-project} \href{https://stacks.math.columbia.edu/tag/0BNG}{0BNG}). The conclusion follows from \ref{lem:hyodo-K0diff} and the isomorphism $\ca{O}_K\otimes_{\ca{O}_{K_0}} (\oplus_{i\in I}\ca{O}_{K_0})^\wedge\iso (\oplus_{i\in I}\ca{O}_{K})^\wedge$ as $\ca{O}_K$ is finite free over $\ca{O}_{K_0}$.
\end{proof}

\begin{mylem}[\cite{hyodo1986hodge} 4-4]\label{lem:hyodo-M0diff}
	Let $M_0=\bigcup_{i\in I,m\geq 0}K_0(w_{im})\subseteq \overline{K}$.
	Then there is an isomorphism of $\ca{O}_{M_0}$-modules
	\begin{align}\label{eq:hyodo-M0diff}
	M_0\otimes_{\ca{O}_{K_0}}(\oplus_{i\in I}\ca{O}_{K_0})^\wedge\iso \omhat{K_0}{M_0},\  p^{-m}\otimes e_i\mapsto \df\log w_{im},\ \forall i\in I,m\in \bb{N}.
	\end{align}
\end{mylem}
\begin{proof}
	For an integer $N>0$ and a finite subset $J\subseteq I$, let $L_0=\bigcup_{i\in J}K_0(w_{iN})$. Then by \ref{lem:hyodo-K0diff}, $(\oplus_{i\in I}\ca{O}_{L_0})^\wedge$ is isomorphic to $\widehat{\Omega}_{\ca{O}_{L_0}}^1$ by sending $e_i$ to $\df\log w_{iN}$ if $i\in J$, and to $\df\log u_i$ if $i\notin J$. The conclusion follows by taking colimit over $J$ and $N$.
\end{proof}

\begin{mylem}[\cite{hyodo1986hodge} 4-7]\label{lem:hyodo-M/M0}
	With the same notation as in {\rm\ref{lem:hyodo-M0diff}}, let $M$ be a finite extension of $M_0$. 
	Then there is a canonical exact sequence of $\ca{O}_M$-modules
	\begin{align}\label{eq:hyodo-M/M0}
	\xymatrix{
		0\ar[r]&\ca{O}_{M}\otimes_{\ca{O}_{M_0}}\omhat{K_0}{M_0}\ar[r]&\omhat{K_0}{M}\ar[r]& \Omega_{\ca{O}_M/\ca{O}_{M_0}}^1\ar[r]&0.
	}
	\end{align}
\end{mylem}
\begin{proof}
		We notice that $\ca{O}_{M_0}$ is a henselian discrete valuation ring with perfect residue field. Let $M_{\op{ur}}$ be the maximal unramified subextension of $M/M_0$, $f\in \ca{O}_{M_{\op{ur}}}[X]$ the monic minimal polynomial of a uniformizer $\varpi$ of $\ca{O}_M$. Then we have $\ca{O}_M=\ca{O}_{M_{\op{ur}}}[X]/(f(X))$. For a sufficiently large finite subextension $L_1$ of $M_{\op{ur}}/K_0$ such that $f\in \ca{O}_{L_1}[X]$, $L_2=L_1(\varpi)$ is totally ramified over $L_1$. The same argument as in \ref{prop:hyodo-K/Z-diff} gives us a canonical exact sequence
	\begin{align}
	\xymatrix{
		0\ar[r]&\ca{O}_{L_2}\otimes_{\ca{O}_{L_1}}\widehat{\Omega}_{\ca{O}_{L_1}}^1 \ar[r]&\widehat{\Omega}_{\ca{O}_{L_2}}^1\ar[r]& \Omega_{\ca{O}_{L_2}/\ca{O}_{L_1}}^1\ar[r]&0.
	}
	\end{align}
	By taking colimit over $L_1$, we get an exact sequence
	\begin{align}\label{eq:2.6.3}
	\xymatrix{
		0\ar[r]&\ca{O}_{M}\otimes_{\ca{O}_{M_{\op{ur}}}}\omhat{K_0}{M_{\op{ur}}}\ar[r]&\omhat{K_0}{M}\ar[r]& \Omega_{\ca{O}_M/\ca{O}_{M_{\op{ur}}}}^1\ar[r]&0.
	}
	\end{align}
	A similar colimit argument shows that $\omhat{K_0}{M_{\op{ur}}}=\ca{O}_{M_{\op{ur}}}\otimes_{\ca{O}_{M_0}}\omhat{K_0}{M_{0}}$. The conclusion follows from \eqref{eq:2.6.3}.
\end{proof}

\begin{myprop}[\cite{hyodo1986hodge} 4-2-2]\label{prop:hyodo-Kbar/Z}
	There is an exact sequence of $\ca{O}_{\overline{K}}$-$G_K$-modules which splits as a sequence of $\ca{O}_{\overline{K}}$-modules,
	\begin{align}\label{eq:hyodo-Kbar/Z}
	\xymatrix{
		0\ar[r]&\overline{K}/\ak{a}(1)\ar[r]^-{\vartheta}&\omhat{K}{\overline{K}}\ar[r]& \overline{K}\otimes_{\ca{O}_K}(\oplus_{i\in I}\ca{O}_K)^\wedge\ar[r]&0,
	}
	\end{align}
	where $\ak{a}=\{x\in \overline{K}\ |\ v_p(x)\geq-1/(p-1)\}$, and $\vartheta(p^{-k}\otimes (\zeta_n)_n)=\df\log \zeta_k$ for any $k\in \bb{N}$ and any $(\zeta_n)_n\in \bb{Z}_p(1)$. The map $\overline{K}\otimes_{\ca{O}_K}(\oplus_{i\in I}\ca{O}_K)^\wedge\to \omhat{K}{\overline{K}}$, sending $p^{-m}\otimes e_i$ to $\df\log w_{im}$ for any $i\in I$ and $m\in \bb{N}$, gives a splitting of the sequence.
\end{myprop}
\begin{proof}
	With the same notation as in {\rm\ref{lem:hyodo-M0diff}}, let $M$ run through all finite subextensions of $\overline{K}/M_0$. We get from \ref{lem:hyodo-M/M0} an exact sequence of $\ca{O}_{\overline{K}}$-modules
	\begin{align}\label{eq:2.4.6}
	\xymatrix{
		0\ar[r]&\ca{O}_{\overline{K}}\otimes_{\ca{O}_{M_0}}\omhat{K_0}{M_0}\ar[r]&\omhat{K}{\overline{K}}\ar[r]& \Omega_{\ca{O}_{\overline{K}}/\ca{O}_{M_0}}^1\ar[r]&0.
	}
	\end{align}
	We identify its first term with $\overline{K}\otimes_{\ca{O}_K}(\oplus_{i\in I}\ca{O}_K)^\wedge$ by \ref{lem:hyodo-M0diff}. Let $\overline{\bb{Q}_p}$ be the algebraic closure of $\bb{Q}_p$ in $\overline{K}$, $\overline{\bb{Z}_p}$ the integral closure of $\bb{Z}_p$ in $\overline{\bb{Q}_p}$. By Fontaine's computation (\cite{fontaine1982formes}, Th\'{e}or\`{e}me 1$'$), we have an isomorphism of $\overline{\bb{Z}_p}$-modules
	\begin{align}\label{eq:3.6.3.1}
	\overline{\bb{Q}_p}/\ak{a}_0(1)\iso \Omega_{\overline{\bb{Z}_p}/\bb{Z}_p}^1,\ p^{-k}\otimes (\zeta_n)_n\mapsto \df\log \zeta_k,\ \forall k\in \bb{N},
	\end{align}
	where $\ak{a}_0=\{x\in \overline{\bb{Q}_p}\ |\ v_p(x)\geq -1/(p-1)\}$, and we have an isomorphism of $\ca{O}_{\overline{K}}$-modules
	\begin{align}\label{eq:3.6.3}
	\overline{K}/\ak{a}(1)\iso\Omega_{\ca{O}_{\overline{K}}/\ca{O}_{M_0}}^1,\ p^{-k}\otimes (\zeta_n)_n\mapsto \df\log \zeta_k,\ \forall k\in \bb{N},
	\end{align}
	where $\ak{a}=\{x\in \overline{K}\ |\ v_p(x)\geq -1/(p-1)\}$. Hence, the composition of
	\begin{align}
	\overline{K}/\ak{a}(1)\iso\ca{O}_{\overline{K}}\otimes_{\overline{\bb{Z}_p}}\Omega_{\overline{\bb{Z}_p}/\bb{Z}_p}^1\longrightarrow \Omega_{\ca{O}_{\overline{K}}/\bb{Z}_p}^1\longrightarrow \omhat{K}{\overline{K}}
	\end{align} 
	gives a splitting of \eqref{eq:2.4.6}. Thus, we obtain the splitting sequence \eqref{eq:hyodo-Kbar/Z} of $\ca{O}_{\overline{K}}$-modules. We notice that the Galois conjugates of $\zeta_n, w_{im}$ are of the form $\zeta_n^a,\zeta_m^bw_{im}$ respectively, which implies that \eqref{eq:hyodo-Kbar/Z} is $G_K$-equivariant.
\end{proof}

\subsection{}
As $\omhat{K}{\overline{K}}$ is $p$-divisible, we have an exact sequence $0\to T_p(\omhat{K}{\overline{K}})\to V_p(\omhat{K}{\overline{K}})\to \omhat{K}{\overline{K}}\to 0$. After inverting $p$, we get an exact sequence
\begin{align}\label{eq:hyodo-fals}
\xymatrix{
	0\ar[r]& C(1)\ar[r]& \overline{K}\otimes_{\ca{O}_{\overline{K}}}V_p(\omhat{K}{\overline{K}})\ar[r]&\overline{K}\otimes_{\ca{O}_{\overline{K}}}\omhat{K}{\overline{K}}\ar[r]&0,
}
\end{align}
where we identified $\overline{K}\otimes_{\ca{O}_{\overline{K}}}T_p(\omhat{K}{\overline{K}})$ with $C(1)$ by \eqref{eq:hyodo-Kbar/Z}.

\begin{mythm}[\cite{hyodo1986hodge} Theorem 1 and Remark 3]\label{thm:hyodo-galcoh}\mbox{ }
	\begin{enumerate}[\rm (i)]
		\item The composition of
		\begin{align}\label{eq:hyodo-connmap}
		\xymatrix{
			K\otimes_{\ca{O}_K}\widehat{\Omega}_{\ca{O}_K}^1\ar[r]^-{\epsilon}&(\overline{K}\otimes_{\ca{O}_{\overline{K}}}\omhat{K}{\overline{K}})^{G_K}\ar[r]^-{\delta}&H^1(G_K,C(1)),
		}
		\end{align}
		where $\epsilon$ is the canonical map and $\delta$ is the connecting map associated to \eqref{eq:hyodo-fals}, is an isomorphism. Moreover, for any integer $q$, the cup product induces an isomorphism
		\begin{align}\label{eq:hyodo-galcoh2}
		(\wedge^q H^1(G_K,C(1)))^\wedge\iso H^q(G_K,C(q)).
		\end{align}
		\item The $K$-module $H^1(G_K,C)$ is free of rank $1$. Moreover, for any integer $q$, the cup product induces an isomorphism
		\begin{align}\label{eq:hyodo-galcoh3}
		H^1(G_K,C)\otimes_K (\wedge^{q-1} H^1(G_K,C(1)))^\wedge\iso H^q(G_K,C(q-1)).
		\end{align}
		\item For any integers $r$ and $q$ such that $r\neq q$ or $q-1$, we have $H^q(G_K,C(r))=0$.
	\end{enumerate}
\end{mythm}

\begin{myrem}\label{rem:hyodo-H1}
	By \ref{prop:hyodo-K/Z-diff} we have an isomorphism
	\begin{align}\label{eq:KhatOmegaK}
	K\otimes_{\ca{O}_K} (\oplus_{i\in I}\ca{O}_K)^\wedge\iso K\otimes_{\ca{O}_K}\widehat{\Omega}_{\ca{O}_K}^1,\ 1\otimes e_i\mapsto 1\otimes \df\log u_i,\ \forall i\in I.
	\end{align}
	By composing it with \eqref{eq:hyodo-connmap}, we get an isomorphism
	\begin{align}\label{eq:H1}
	K\otimes_{\ca{O}_K} (\oplus_{i\in I}\ca{O}_K)^\wedge\iso H^1(G_K,C(1)),\ 1\otimes e_i\mapsto [f_i],
	\end{align}
	where $f_i$ is a $1$-cocycle sending each $\sigma\in G_K$ to $\sigma(1\otimes(\df\log w_{im})_m)-1\otimes(\df\log w_{im})_m$ in view of \eqref{eq:hyodo-fals}.
\end{myrem}

\section{Faltings Extension}\label{sec:falext}
\begin{mylem}\label{lem:M/K-diff}
	Let $M=\bigcup_{i\in I,m\geq 0}K(w_{im})\subseteq \overline{K}$. Then there is an isomorphism of $\ca{O}_M$-modules
	\begin{align}\label{eq:M/K-diff}
	\oplus_{i\in I} M/\ca{O}_M\iso \Omega_{\ca{O}_M/\ca{O}_K}^1 ,\ p^{-m}e_i\mapsto \df\log w_{im},\ \forall i\in I,m\in \bb{N}.
	\end{align}
\end{mylem}
\begin{proof}
	 For any $N\geq 0$, we set $M_N=\bigcup_{i\in I}K(w_{iN})$. Since $(\overline{u_i})$ form a $p$-base of the residue field $k$, the elements of the form $\prod_{i\in I} \overline{w_{iN}}^{k_i}$ where $0\leq k_i<p^N$ with finitely many nonvanishing, are linearly independent over $k$. Therefore, $\ca{O}_{M_N}=\ca{O}_K[T_i]_{i\in I}/(T_i^{p^N}-u_i)$, where $T_i$ maps to $w_{iN}$.
	 Hence,
	 \begin{align}
	 \Omega_{\ca{O}_{M_N}/\ca{O}_K}^1=\oplus_{i\in I}\ca{O}_{M_N}/p^N\ca{O}_{M_N}=\oplus_{i\in I}p^{-N}\ca{O}_{M_N}/\ca{O}_{M_N},
	 \end{align}
	 where $p^{-N}e_i$ corresponds to $\df\log w_{iN}$. The conclusion follows by taking colimit over $N$.
\end{proof}

\begin{myprop}\label{prop:falext-M/K}
	With the same notation as in {\rm \ref{lem:M/K-diff}}, there is an exact sequence of $\ca{O}_{\overline{K}}$-modules
	\begin{align}\label{eq:falext-M/K}
	\xymatrix{
		0\ar[r]& \oplus_{i\in I} \overline{K}/\ca{O}_{\overline{K}}\ar[r]^-{\theta}& \Omega_{\ca{O}_{\overline{K}}/\ca{O}_K}^1\ar[r]&\overline{K}/\ak{b}(1)\ar[r]&0,
	}
	\end{align}
	where $\theta(p^{-m} e_i)=\df\log w_{im}$ for any $i\in I$ and $m\in\bb{N}$, and $\ak{b}=\{x\in \overline{K}\ |\ v_p(x)\geq -v_p(\ak{D}_{M/M_1})-1/(p-1)\}$, where $M_1$ is the fraction field of the Witt ring with coefficients in the residue field of $M$, and $\ak{D}_{M/M_1}$ is the different ideal of $M/M_1$.
\end{myprop}
\begin{proof}
	 We notice that $\ca{O}_M$ is a henselian discrete valuation ring with perfect residue field. Thus, the sequence of modules of differentials of $\ca{O}_{\overline{K}}/\ca{O}_M/\ca{O}_K$,
	\begin{align}\label{eq:falext-M/K2}
	\xymatrix{
		0\ar[r]& \ca{O}_{\overline{K}}\otimes_{\ca{O}_M}\Omega_{\ca{O}_M/\ca{O}_K}^1\ar[r]& \Omega_{\ca{O}_{\overline{K}}/\ca{O}_K}^1\ar[r]&\Omega_{\ca{O}_{\overline{K}}/\ca{O}_M}^1\ar[r]&0,
	}
	\end{align}
	is exact by \ref{lem:dvr-diffinj}. We identify its first term with $\oplus_{i\in I} \overline{K}/\ca{O}_{\overline{K}}$ by \ref{lem:M/K-diff}. By Fontaine's computation (\cite{fontaine1982formes}, Th\'{e}or\`{e}me 1$'$), we have an isomorphism of $\ca{O}_{\overline{K}}$-modules
	\begin{align}\label{eq:Kbar/M}
	\overline{K}/\ak{b}(1)\iso \Omega_{\ca{O}_{\overline{K}}/\ca{O}_M}^1,\ p^{-k}\otimes (\zeta_n)_n\mapsto \df\log \zeta_k,\ \forall k\in \bb{N},\ \forall (\zeta_n)_n\in \bb{Z}_p(1).
	\end{align}
	The conclusion follows from \eqref{eq:falext-M/K2}.
\end{proof}

\begin{mylem}\label{lem:G-invCI}
	The canonical map
	\begin{align}\label{eq:G-invCI}
	K\otimes_{\ca{O}_K}(\oplus_{i\in I}\ca{O}_K)^\wedge\longrightarrow (C\otimes_{\ca{O}_C}(\oplus_{i\in I}\ca{O}_C)^\wedge)^{G_K}
	\end{align}
	is an isomorphism.
\end{mylem}
\begin{proof}
	It follows from the following descriptions
	\begin{align}
	\label{eq:sum-i-comp1}C\otimes_{\ca{O}_C}(\oplus_{i\in I}\ca{O}_C)^\wedge=&\{(x_i)\in\prod\nolimits_{i\in I}C\ |\ \forall N>0,\ \exists \text{ finite }J\subseteq I, |x_i|_p<1/N,\ \forall i\notin J\},\\
	\label{eq:sum-i-comp2}K\otimes_{\ca{O}_K}(\oplus_{i\in I}\ca{O}_K)^\wedge=&\{(x_i)\in\prod\nolimits_{i\in I}K\ |\ \forall N>0,\ \exists \text{ finite }J\subseteq I, |x_i|_p<1/N,\ \forall i\notin J\}.
	\end{align}
\end{proof}

\begin{mythm}\label{thm:falext}
	There is a canonical exact sequence of $C$-$G_K$-modules which splits as a sequence of $C$-modules,
	\begin{align}\label{eq:can-falext}
	\xymatrix{
		0\ar[r]& C(1)\ar[r]^-{\iota}&V_p(\Omega_{\ca{O}_{\overline{K}}/\ca{O}_K}^1)\ar[r]^-{\nu}&C\otimes_{\ca{O}_C}(\ca{O}_{\overline{K}}\otimes_{\ca{O}_K} \Omega_{\ca{O}_K/\bb{Z}_p}^1)^\wedge\ar[r]&0,
	}
	\end{align}
	where $\iota(1\otimes(\zeta_n)_n)=(\df\log\zeta_n)_n$ for any $(\zeta_n)_n\in \bb{Z}_p(1)$. There is an isomorphism of $C$-$G_K$-modules
	\begin{align}\label{eq:4.4.2}
	C\otimes_{\ca{O}_C}(\oplus_{i\in I}\ca{O}_C)^\wedge \iso C\otimes_{\ca{O}_C}(\ca{O}_{\overline{K}}\otimes_{\ca{O}_K} \Omega_{\ca{O}_K/\bb{Z}_p}^1)^\wedge,\ 1\otimes e_i\mapsto 1\otimes 1\otimes \df\log u_i,\ \forall i\in I,
	\end{align}
	and the map $C\otimes_{\ca{O}_C}(\oplus_{i\in I}\ca{O}_C)^\wedge\to V_p(\Omega_{\ca{O}_{\overline{K}}/\ca{O}_K}^1)$, sending $1\otimes e_i$ to $(\df\log w_{im})_m$ for any $i\in I$, gives a $C$-linear section of $\nu$.	
\end{mythm}
\begin{proof}
	We consider the sequence of modules of differentials of $\ca{O}_L/\ca{O}_K/\bb{Z}_p$, where $L/K$ is a finite subextension of $\overline{K}/K$, and pass to $p$-adic completions. Since $\Omega_{\ca{O}_{L}/\ca{O}_K}^1$ is killed by a power of $p$, we still get an exact sequence (\cite{stacks-project} \href{https://stacks.math.columbia.edu/tag/0315}{0315}, \href{https://stacks.math.columbia.edu/tag/0BNG}{0BNG})
	\begin{align}\label{eq:can1}
	\xymatrix{
		\ca{O}_L\otimes_{\ca{O}_K}\widehat{\Omega}_{\ca{O}_K}^1\ar[r]& \widehat{\Omega}_{\ca{O}_L}^1\ar[r]&\Omega_{\ca{O}_{L}/\ca{O}_K}^1\ar[r]&0.
	}
	\end{align}
	By taking colimit over all such $L$, we get an exact sequence 
	\begin{align}\label{eq:Kbar-K-Z-diff}
	\xymatrix{
		\ca{O}_{\overline{K}}\otimes_{\ca{O}_K}\widehat{\Omega}_{\ca{O}_K}^1\ar[r]^-{\alpha}& \omhat{K}{\overline{K}}\ar[r]^-{\beta}&\Omega_{\ca{O}_{\overline{K}}/\ca{O}_K}^1\ar[r]&0.
	}
	\end{align}
	Combining with propositions \ref{prop:hyodo-K/Z-diff}, \ref{prop:hyodo-Kbar/Z} and \ref{prop:falext-M/K}, we get a commutative diagram:
	\begin{align}\label{diam:can}
	\xymatrix{
		0\ar[r]&\ca{O}_{\overline{K}}\otimes_{\ca{O}_K}(\oplus_{i\in I}\ca{O}_K)^\wedge\ar@{^{(}->}[d]\ar[r]&\overline{K}\otimes_{\ca{O}_K}(\oplus_{i\in I}\ca{O}_K)^\wedge\ar@{^{(}->}[d]\ar[r]&\oplus_{i\in I}\overline{K}/\ca{O}_{\overline{K}}\ar@{^{(}->}[d]\ar[r]&0\\
		&\ca{O}_{\overline{K}}\otimes_{\ca{O}_K}\widehat{\Omega}_{\ca{O}_K}^1\ar[r]^-{\alpha}\ar@{->>}[d]& \omhat{K}{\overline{K}}\ar[r]^-{\beta}\ar@{->>}[d]&\Omega_{\ca{O}_{\overline{K}}/\ca{O}_K}^1\ar[r]\ar@{->>}[d]&0\\
		&\ca{O}_{\overline{K}}\otimes_{\ca{O}_K}\Omega_{\ca{O}_K/\ca{O}_{K_0}}^1&\overline{K}/\ak{a}(1)\ar[r]&\overline{K}/\ak{b}(1)\ar[r]&0
	}
	\end{align}
	where the rows and columns are exact, and the middle column splits. We set $D=\ke(\beta)=\im(\alpha)$. We see that $\ca{O}_{\overline{K}}\otimes_{\ca{O}_K}(\oplus_{i\in I}\ca{O}_K)^\wedge\to D$ is injective, whose cokernel is killed by a power of $p$.
	Now for any $n>0$, by applying $\ho_{\bb{Z}_p}(\bb{Z}_p/p^n\bb{Z}_p,-)$ to \eqref{eq:Kbar-K-Z-diff}, we get an exact sequence of $\ca{O}_{\overline{K}}/p^n\ca{O}_{\overline{K}}$-modules
	\begin{align}\label{eq:can2}
	\xymatrix@C=15pt{
		0\ar[r]&D[p^n]\ar[r]& \omhat{K}{\overline{K}}[p^n]\ar[r]&\Omega_{\ca{O}_{\overline{K}}/\ca{O}_K}^1[p^n]\ar[r]&D/p^nD\ar[r]&\omhat{K}{\overline{K}}/p^n\omhat{K}{\overline{K}}=0.
	}
	\end{align}
	We notice that the inverse system $(D[p^n])_n$ is Artin-Rees null, and that $(\omhat{K}{\overline{K}}[p^n])_n$ satisfies Mittag-Leffler condition. Therefore, by taking the inverse limit of \eqref{eq:can2}, we get an exact sequence of $\ca{O}_{C}$-modules
	\begin{align}\label{eq:intcan-falext}
	\xymatrix{
		0\ar[r]& T_p(\omhat{K}{\overline{K}})\ar[r]&T_p(\Omega_{\ca{O}_{\overline{K}}/\ca{O}_K}^1)\ar[r]&D^\wedge\ar[r]&0.
	}
	\end{align}
	By applying $T_p(-)$ to the middle column of \eqref{diam:can}, we get $T_p(\omhat{K}{\overline{K}})=\widehat{\ak{a}}(1)$.
	On the other hand, we notice that $\oplus_{i\in I}\overline{K}/\ca{O}_{\overline{K}}$ is $p$-divisible, and that $((\oplus_{i\in I}\overline{K}/\ca{O}_{\overline{K}})[p^n])_n$ satisfies Mittag-Leffler condition. Therefore, by applying $T_p(-)$ to the right column of (\ref{diam:can}), we get an exact sequence of $\ca{O}_C$-modules
	\begin{align}\label{eq:falext-re}
	\xymatrix@R=1pt{
		0\ar[r]&(\oplus_{i\in I}\ca{O}_C)^\wedge\ar[r]& T_p(\Omega_{\ca{O}_{\overline{K}}/\ca{O}_K}^1)\ar[r]&\widehat{\ak{b}}(1)\ar[r]&0.
	}
	\end{align}
	
	As $\Omega_{\ca{O}_K/\ca{O}_{K_0}}^1$ is killed by a power of $p$, the map $(\ca{O}_{\overline{K}}\otimes_{\ca{O}_K}\widehat{\Omega}_{\ca{O}_K}^1)^\wedge\to D^\wedge$ becomes an isomorphism after inverting $p$. Afterwards, we get from \eqref{eq:intcan-falext} a canonical exact sequence of $C$-modules
	\begin{align}\label{eq:can-falext2}
	\xymatrix{
		0\ar[r]& C(1)\ar[r]&V_p(\Omega_{\ca{O}_{\overline{K}}/\ca{O}_K}^1)\ar[r]&C\otimes_{\ca{O}_C}(\ca{O}_{\overline{K}}\otimes_{\ca{O}_K} \widehat{\Omega}_{\ca{O}_K}^1)^\wedge\ar[r]&0,
	}
	\end{align}
	and from \eqref{eq:falext-re} an exact sequence of $C$-modules
	\begin{align}\label{eq:falext-re2}
	\xymatrix@R=1pt{
		0\ar[r]&C\otimes_{\ca{O}_C}(\oplus_{i\in I}\ca{O}_C)^\wedge\ar[r]& V_p(\Omega_{\ca{O}_{\overline{K}}/\ca{O}_K}^1)\ar[r]&C(1)\ar[r]&0.
	}
	\end{align}
	The latter gives a splitting of \eqref{eq:can-falext2} and an isomorphism $C\otimes_{\ca{O}_C}(\oplus_{i\in I}\ca{O}_C)^\wedge\iso C\otimes_{\ca{O}_C}(\ca{O}_{\overline{K}}\otimes_{\ca{O}_K} \widehat{\Omega}_{\ca{O}_K}^1)^\wedge$ by sending $1\otimes e_i$ to $1\otimes 1\otimes \df\log u_i$ by diagram chasing.
	We notice that the Galois conjugates of $\zeta_n, w_{im}$ are of the form $\zeta_n^a,\zeta_m^bw_{im}$ respectively, which implies that \eqref{eq:can-falext2} is $G_K$-equivariant. Hence, \eqref{eq:can-falext2} gives us the exact sequence \eqref{eq:can-falext} of $C$-$G_K$-modules which splits as a sequence of $C$-modules.
\end{proof}

\begin{mycor}\label{cor:falext-conn}
	The canonical map $K\otimes_{\ca{O}_K}\widehat{\Omega}_{\ca{O}_K}^1\to (C\otimes_{\ca{O}_C}(\ca{O}_{\overline{K}}\otimes_{\ca{O}_K} \Omega_{\ca{O}_K/\bb{Z}_p}^1)^\wedge)^{G_K}$ is an isomorphism, and the connecting map of the sequence \eqref{eq:can-falext}
	\begin{align}\label{eq:falext-conn}
	\delta: K\otimes_{\ca{O}_K}\widehat{\Omega}_{\ca{O}_K}^1\longrightarrow H^1(G_K,C(1))
	\end{align}
	is an isomorphism which coincides with \eqref{eq:hyodo-connmap}. In particular,
	\begin{align}\label{eq:falext-Ginv}
	V_p(\Omega_{\ca{O}_{\overline{K}}/\ca{O}_K}^1)^{G_K}=0.
	\end{align}
\end{mycor}
\begin{proof}
	By \eqref{eq:KhatOmegaK}, \eqref{eq:4.4.2} and \ref{lem:G-invCI}, we see that the canonical map $K\otimes_{\ca{O}_K}\widehat{\Omega}_{\ca{O}_K}^1\to (C\otimes_{\ca{O}_C}(\ca{O}_{\overline{K}}\otimes_{\ca{O}_K} \widehat{\Omega}_{\ca{O}_K}^1)^\wedge)^{G_K}$ is an isomorphism.
	Now \eqref{eq:falext-conn} follows from \ref{thm:hyodo-galcoh} (\luoma{1}) and \ref{rem:hyodo-H1}. And \eqref{eq:falext-Ginv} follows from the fact that $C(1)^{G_K}=0$.
\end{proof}

\begin{mydefn}\label{defn:fal-ext}
	We call the sequence \eqref{eq:can-falext} the \emph{Faltings extension of $\ca{O}_K$ over $\bb{Z}_p$}.
\end{mydefn}

\section{Fontaine's Injection}\label{sec:foninj}
\subsection{}
For any proper model $\ak{X}$ of the abelian variety $X$ over $\ca{O}_K$ (i.e. a proper $\ca{O}_K$-scheme whose generic fiber is $X$), we identify $\ak{X}(\ca{O}_{\overline{K}})$ with $X(\overline{K})$ by valuative criterion. Pullback of K\"ahler differentials defines a map
\begin{align}\label{eq:pullback-diff}
H^0(\ak{X},\Omega_{\ak{X}/\ca{O}_K}^1)\longrightarrow \op{Map}_{G_K}(X(\overline{K}),\Omega_{\ca{O}_{\overline{K}}/\ca{O}_K}^1),\ \omega\mapsto (u\mapsto u^*\omega).
\end{align}
We notice that $H^0(X,\Omega_{X/K}^1)=K\otimes_{\ca{O}_K}H^0(\ak{X},\Omega_{\ak{X}/\ca{O}_K}^1)$, and that any differential form over $X$ is invariant under translations. Hence, we can take an integer $r>0$ big enough, such that for any $\omega\in p^r H^0(\ak{X},\Omega_{\ak{X}/\ca{O}_K}^1)$ and $u_1,u_2\in \ak{X}(\ca{O}_{\overline{K}})$, $(u_1+u_2)^*\omega=u_1^* \omega+u_2^*\omega$ (cf. \cite{fontaine1982formes} Proposition 3). Therefore, \eqref{eq:pullback-diff} induces a homomorphism of $\ca{O}_K$-modules
\begin{align}\label{eq:pullback-invdiff}
\rho_1:p^rH^0(\ak{X},\Omega_{\ak{X}/\ca{O}_K}^1)\longrightarrow \ho_{\bb{Z}[G_K]}(X(\overline{K}),\Omega_{\ca{O}_{\overline{K}}/\ca{O}_K}^1),\ \omega\mapsto (u\mapsto u^*\omega).
\end{align}
We may also assume that $p^rH^0(\ak{X},\Omega_{\ak{X}/\ca{O}_K}^1)$ has no $p$-torsion for further use.

\subsection{}
The functor $V_p(-)$ gives us an injective homomorphism
\begin{align}\label{eq:rho2}
\rho_2:\ho_{\bb{Z}[G_K]}(X(\overline{K}),\Omega_{\ca{O}_{\overline{K}}/\ca{O}_K}^1)\longrightarrow \ho_{\bb{Z}[G_K]}(V_p(X),V_p(\Omega_{\ca{O}_{\overline{K}}/\ca{O}_K}^1))
\end{align}
since $X(\overline{K})$ is $p$-divisible (cf. \cite{fontaine1982formes} 3.5 Lemme 1).

\subsection{}
The composition $\rho_2\circ \rho_1$ induces a homomorphism of $K$-modules
\begin{align}\label{eq:rho2-rho1}
H^0(X,\Omega_{X/K}^1)=K\otimes_{\ca{O}_K}p^rH^0(\ak{X},\Omega_{\ak{X}/\ca{O}_K}^1)\longrightarrow \ho_{\bb{Z}[G_K]}(V_p(X),V_p(\Omega_{\ca{O}_{\overline{K}}/\ca{O}_K}^1)).
\end{align}
As the category of $\ca{O}_K$-proper models of $X$ is connected, this composition does not depend on the choice of the model and number $r$ (cf. \cite{fontaine1982formes} Proposition 4). 
We conclude by the following lemma that \eqref{eq:rho2-rho1} is injective.
\begin{mylem}[\cite{fontaine1982formes} 3.5 Lemme 1]\label{lem:rho1inj}
	There is a proper model $\ak{X}$ of $X$ such that $\rho_1$ is injective.
\end{mylem}
\begin{proof}
	We follow closely the proof of (\cite{fontaine1982formes} 3.5 Lemme 1), which does not essentially use the assumption that the residue field $k$ is perfect. We briefly sketch how to adapt Fontaine's proof.
	
	(1) Let $u$ be the origin of $X$ and $d$ the dimension of $X$. We first take a closed immersion $X\to \bb{P}^n_K$, and then we take an open immersion $\bb{P}^n_K\to\bb{P}^n_{\ca{O}_K}$ described later (all the morphisms are over $\ca{O}_K$). Let $\ak{X}$ be the scheme theoretic image of the composition $X\to \bb{P}^n_{\ca{O}_K}$, which is thus a proper model of $X$. Let $\overline{u}$ be the special point of the scheme theoretic image of $u$. It is a $k$-point. After a linear transformation of coordinates, we can at first choose an open immersion $\bb{P}^n_K\to \bb{P}^n_{\ca{O}_K}$ such that $\ca{O}_{\ak{X},\overline{u}}$ is a $(d+1)$-dimensional regular local ring (cf. \cite{fontaine1982formes} 3.6 Lemme 3).
	
	(2) The $\ak{m}_{\ak{X},\overline{u}}$-adic completion of the local ring $\ca{O}_{\ak{X},\overline{u}}$ is isomorphic to $\ca{O}_K\llbracket T_1,\dots,T_d\rrbracket$, denoted by $\widehat{\ca{O}}_{\ak{X},\overline{u}}$. The $\ak{m}_{\ak{X},\overline{u}}$-adic completion of $\Omega_{\ca{O}_{\ak{X},\overline{u}}/\ca{O}_K}^{1}$ is a free $\widehat{\ca{O}}_{\ak{X},\overline{u}}$-module of rank $d$, denoted by $\widehat{\Omega}_{\ca{O}_{\ak{X},\overline{u}}/\ca{O}_K}^{1}$. The invariance of differential forms over $X$ and the fact that $p^rH^0(\ak{X},\Omega_{\ak{X}/\ca{O}_K}^1)\subseteq H^0(X,\Omega_{X/K}^1)$ imply that the canonical map $p^rH^0(\ak{X},\Omega_{\ak{X}/\ca{O}_K}^1)\to {\Omega}_{\ca{O}_{\ak{X},\overline{u}}/\ca{O}_K}^{1}$ is injective (cf. \cite{fontaine1982formes} 3.7). We remark that the canonical map ${\Omega}_{\ca{O}_{\ak{X},\overline{u}}/\ca{O}_K}^{1}\to \widehat{\Omega}_{\ca{O}_{\ak{X},\overline{u}}/\ca{O}_K}^{1}$ is injective as ${\Omega}_{\ca{O}_{\ak{X},\overline{u}}/\ca{O}_K}^{1}$ is of finite type over the Noetherian local ring $\ca{O}_{\ak{X},\overline{u}}$.
	
	(3) We have the following commutative diagram
	\begin{align}\label{diam:local-diff}
	\xymatrix{
		p^rH^0(\ak{X},\Omega_{\ak{X}/\ca{O}_K}^1)\ar@{^{(}->}[d]\ar[r]^-{\rho_1}&\ho_{\bb{Z}[G_K]}(X(\overline{K}),\Omega_{\ca{O}_{\overline{K}}/\ca{O}_K}^1)\ar[d]\\
		\widehat{\Omega}_{\ca{O}_{\ak{X},\overline{u}}/\ca{O}_K}^{1}\ar[r]^-{\rho_1'}&\op{Map}(\ho_{\ca{O}_K\text{-}\op{cont}}(\widehat{\ca{O}}_{\ak{X},\overline{u}},\ca{O}_{\overline{K}}),\Omega_{\ca{O}_{\overline{K}}/\ca{O}_K}^1)
	}
	\end{align}
	where we identify the set of continuous $\ca{O}_K$-algebra homomorphisms from $\widehat{\ca{O}}_{\ak{X},\overline{u}}$ to $\ca{O}_{\overline{K}}$ with a subset of $\ak{X}(\ca{O}_{\overline{K}})=X(\overline{K})$. To show the injectivity of $\rho_1$ it suffices to show that of $\rho_1'$. More precisely, we need to show that for any nonzero formal differential form $\sum_{i=1}^d\alpha_i(T_1,\dots,T_d)\df T_i$ where $\alpha_i\in \ca{O}_K\llbracket T_1,\dots,T_d\rrbracket$, there are $x_1,\dots,x_d\in \ak{m}_{\overline{K}}$ such that $\sum_{i=1}^d\alpha_i(x_1,\dots,x_d)\df x_i$ is not zero in $\Omega_{\ca{O}_{\overline{K}}/\ca{O}_K}^1$.
	
	(4) For $d=1$, suppose $\alpha(T)=\sum_{k\geq 0}a_kT^k$ where $a_k\in \ca{O}_K$ not all zero. Let $k_0$ be the minimal number such that $v_p(a_{k_0})$ is minimal. For a sufficiently large integer $N$, we take $x=\varpi^{1/p^N}\in \ak{m}_{\overline{K}}$, where $\varpi$ is a uniformizer of $\ca{O}_K$, such that $v_p(a_{k_0}x^{k_0})<v_p(a_kx^k)$ for any $k\neq k_0$. Let $M=\bigcup_{i\in I,m\geq 0}K(w_{im})\subseteq \overline{K}$. The annihilator of $\df x$ in $\Omega_{\ca{O}_{M(x)}/\ca{O}_M}^1$ is generated by $p^Nx^{p^N-1}$. As $\ca{O}_M$ is a henselian discrete valuation ring with perfect residue field, lemma \ref{lem:dvr-diffinj} implies that the annihilator of $\df x$ in $\Omega_{\ca{O}_{\overline{K}}/\ca{O}_M}^1$ is also generated by $p^Nx^{p^N-1}$. When $N$ is big enough, $\alpha(x)\df x$ is not zero in $\Omega_{\ca{O}_{\overline{K}}/\ca{O}_K}^1$ (cf. \cite{fontaine1982formes} 3.7 Lemme 4).
	
	(5) As $\ca{O}_K$ is an infinite domain, there are formal series $\beta_1,\dots,\beta_d\in \ca{O}_K\llbracket T\rrbracket$ without constant term, such that $\sum_{i=1}^d \alpha_i(\beta_1,\dots,\beta_d)\cdot \beta_i'\in \ca{O}_K\llbracket T\rrbracket$ is still nonzero. Hence, the general case reduces to the case $d=1$ (cf. \cite{fontaine1982formes} 3.7 Lemme 5).
\end{proof}

\subsection{}\label{sec:rho3}
As $X(\overline{K})$ is $p$-divisible, we have a canonical exact sequence
\begin{align}
\xymatrix{
0\ar[r]&T_p(X)\ar[r]&V_p(X)\ar[r]&X(\overline{K})\ar[r]&0.
}
\end{align}
After applying the functor $\ho_{\bb{Z}[G_K]}(-,V_p(\Omega_{\ca{O}_{\overline{K}}/\ca{O}_K}^1))$, we get an exact sequence
\begin{align}
0\to\ho_{\bb{Z}[G_K]}(X(\overline{K}),V_p(\Omega_{\ca{O}_{\overline{K}}/\ca{O}_K}^1))\to\ho_{\bb{Z}[G_K]}(V_p(X),V_p(\Omega_{\ca{O}_{\overline{K}}/\ca{O}_K}^1))\to\ho_{\bb{Z}[G_K]}(T_p(X),V_p(\Omega_{\ca{O}_{\overline{K}}/\ca{O}_K}^1)).
\end{align}
Let $f:X(\overline{K})\to V_p(\Omega_{\ca{O}_{\overline{K}}/\ca{O}_K}^1)$ be a $G_K$-equivariant homomorphism.  For any finite extension $L/K$, we denote by $G_L=\op{Gal}(\overline{K}/L)$ the absolute Galois group of $L$. Then $f$ maps $X(L)$ to $V_p(\Omega_{\ca{O}_{\overline{K}}/\ca{O}_K}^1)^{G_L}$. We notice that the kernel of the surjection $\Omega_{\ca{O}_{\overline{K}}/\ca{O}_K}^1\to \Omega_{\ca{O}_{\overline{K}}/\ca{O}_L}^1$ is killed by a power of $p$, which indicates that the map $V_p(\Omega_{\ca{O}_{\overline{K}}/\ca{O}_K}^1)\to V_p(\Omega_{\ca{O}_{\overline{K}}/\ca{O}_L}^1)$ is an isomorphism. Now, by applying \eqref{eq:falext-Ginv} to $L$, we get
\begin{align}\label{eq:rho3tech}
V_p(\Omega_{\ca{O}_{\overline{K}}/\ca{O}_K}^1)^{G_L}=V_p(\Omega_{\ca{O}_{\overline{K}}/\ca{O}_L}^1)^{G_L}=0.
\end{align}
Hence $f(X(\overline{K}))=\bigcup_{L/K}f(X(L))=0$, which indicates that we have an injective map (cf. \cite{fontaine1982formes} 3.5 Lemme 2)
\begin{align}\label{eq:rho3}
\rho_3:\ho_{\bb{Z}[G_K]}(V_p(X),V_p(\Omega_{\ca{O}_{\overline{K}}/\ca{O}_K}^1))\longrightarrow\ho_{\bb{Z}[G_K]}(T_p(X),V_p(\Omega_{\ca{O}_{\overline{K}}/\ca{O}_K}^1)).
\end{align}
Remark that any element in the image of $\rho_3\circ \rho_2\circ \rho_1$ is $\bb{Z}_p$-linear.
All in all, we have generalized Fontaine's injection (\cite{fontaine1982formes} Th\'{e}or\`{e}me 2$'$) to the imperfect residue field case.
\begin{mythm}\label{thm:fon-inj}
	There is a canonical $K$-linear injective homomorphism
	\begin{align}\label{eq:foninj}
	\rho:H^0(X,\Omega_{X/K}^1)\longrightarrow \ho_{\bb{Z}_p[G_K]}(T_p(X),V_p(\Omega_{\ca{O}_{\overline{K}}/\ca{O}_K}^1)).
	\end{align}
\end{mythm}

\section{Weak Hodge-Tate Representations}\label{sec:wHT}
\begin{mydefn}\label{defn:weakHTrepn}
	For any $C$-$G_K$-module $V$ of finite dimension, let
	\begin{align}
	0=V_0\subsetneq V_1\subsetneq V_2\subsetneq \cdots \subsetneq V_n=V
	\end{align}
	be a composition series of $V$, i.e. $V_{i+1}/V_i$ is an irreducible $C$-$G_K$-module for any $i$. The set of factors $\{V_{i+1}/V_i\}_{0\leq i<n}$ does not depend on the choice of the composition series by Schreier refinement theorem. We call the multiset
	\begin{align}\label{eq:HTweight}
	\op{wt}(V)=\{r_i\ |\ V_{i+1}/V_i\cong C(r_i),\ 0\leq i<n\}
	\end{align}
	the multiset of \emph{weak Hodge-Tate weights} of $V$. If all the factors are Tate twists of $C$, i.e. $\dim_C V$ equals the cardinality of $\op{wt}(V)$, then we call $V$ a \emph{weak Hodge-Tate $C$-representation of $G_K$}. We denote by $\scr{C}$ the full subcategory of finite-dimensional $C$-$G_K$-modules formed by weak Hodge-Tate representations. 
\end{mydefn}
\begin{myprop}\label{prop:HTweight}
	Let $V$ be a finite-dimensional $C$-$G_K$-module.
	\begin{enumerate}[\rm (i)]
		\item For any short exact sequence of finite-dimensional $C$-$G_K$-modules $0\to V'\to V\to V''\to 0$, we have $\op{wt}(V)=\op{wt}(V')\sqcup \op{wt}(V'')$. In particular, $\scr{C}$ is closed under taking subrepresentation, quotient and extension.
		\item For the dual representation $V^*=\ho_C(V,C)$, we have $\op{wt}(V^*)=-\op{wt}(V)$.
	\end{enumerate}
\end{myprop}
\begin{proof}
	The first assertion follows from the basic properties of composition series. The second assertion follows from the basic fact $C(r)^*=C(-r)$.
\end{proof}
\begin{myprop}\label{prop:C^s}
	For $s\in \bb{N}$ and $r\in \bb{Z}$, the subrepresentations and quotients of $C(r)^{\oplus s}$ in $\scr{C}$ are direct summands of $C(r)^{\oplus s}$ of the form $C(r)^{\oplus t}$ for some $t\in \bb{N}$.
\end{myprop}
\begin{proof}
	After twisting by $-r$, we may assume that $r=0$. For any subrepresentation $V$ of $C^{\oplus s}$, we set $W=C^{\oplus s}/V$. Consider the following commutative diagram
	\begin{align}
	\xymatrix{
		0\ar[r]& V^{G_K}\otimes_KC\ar[r]\ar[d]&C^{\oplus s}\ar[d]\ar[r]&W^{G_K}\otimes_KC\ar[d]&\\
		0\ar[r]&V\ar[r]&C^{\oplus s}\ar[r]&W\ar[r]&0.
	}
	\end{align}
	We see that the first and third vertical maps are injective, because $K$-linearly independent $G_K$-invariant elements are also $C$-linearly independent. But the middle map is identity, which shows that $V=V^{G_K}\otimes_KC$, $W=W^{G_K}\otimes_KC$. Then any splitting of $0\to V^{G_K}\to K^{\oplus s}\to W^{G_K}\to 0$ induces a splitting of $0\to V\to C^{\oplus s}\to W\to 0$, which completes our proof.
\end{proof}

\begin{myprop}\label{prop:repn-split}
	For $s,t\in \bb{N}$ and integers $r_1, r_2$ such that $r_1-r_2\neq 1$ or $0$, any extension of $C(r_2)^{\oplus s}$ by $C(r_1)^{\oplus t}$ in $\scr{C}$ is trivial.
\end{myprop}
\begin{proof}
	After twisting by $-r_2$, we may assume that $r_2=0$ and $r_1=r\neq 1$ or $0$. Given an exact sequence $0\to C(r)^{\oplus t}\to V\to C^{\oplus s}\to 0$, take $G_K$-invariants, then we obtain an exact sequence
	\begin{align}
	\xymatrix{
		0=(C(r)^{\oplus t})^{G_K}\ar[r]&V^{G_K}\ar[r]&K^{\oplus s}\ar[r]&H^1(G_K,C(r)^{\oplus t})=0,
	}
	\end{align}
	from which we get an isomorphism $V^{G_K}\iso K^{\oplus s}$. Hence $V=C(r)^{\oplus t}\oplus C^{\oplus s}$.
\end{proof}

\section{Hodge-Tate Filtration for Abelian Varieties}\label{sec:HTfil}
\subsection{}
We keep the following simplified notation in this section:
\begin{align}\label{eq:simnotation}
&G=G_K,\ \Omega=\Omega_{\ca{O}_{\overline{K}}/\ca{O}_K}^1;\\ &K_I=K\otimes_{\ca{O}_K} \widehat{\Omega}_{\ca{O}_K}^1\stackrel{\sim}{\longleftarrow} K\otimes_{\ca{O}_K}(\oplus_{i\in I}\ca{O}_K)^\wedge\ (\text{by }\eqref{eq:KhatOmegaK});\\ &C_I=C\otimes_{\ca{O}_C}(\ca{O}_{\overline{K}}\otimes_{\ca{O}_K} \Omega_{\ca{O}_K/\bb{Z}_p}^1)^\wedge\stackrel{\sim}{\longleftarrow} C\otimes_{\ca{O}_C}(\oplus_{i\in I}\ca{O}_C)^\wedge\ (\text{by }\eqref{eq:4.4.2});\\
&E=\ho_{\bb{Z}_p}(T_p(X),C),\ E^{G}(1)=\ho_{\bb{Z}_p[G]}(T_p(X),C)\otimes_KC(1)\subseteq E(1).
\end{align}
We remark that the Tate module $T_p(X)$ of the abelian variety $X$ is a finite free $\bb{Z}_p$-module.
By applying the functor $\ho_{\bb{Z}_p}(T_p(X),-)=E\otimes_C-$ to the Faltings extension \eqref{eq:can-falext}, we get an exact sequence of $C$-$G_K$-modules
\begin{align}\label{eq:abvar-falext}
\xymatrix@R=1pt{
	0\ar[r]&\ho_{\bb{Z}_p}(T_p(X),C(1))\ar[r]& \ho_{\bb{Z}_p}(T_p(X),V_p(\Omega))\ar[r]&\ho_{\bb{Z}_p}(T_p(X),C_I)\ar[r]& 0.\\
	&=E(1)&=E\otimes_C V_p(\Omega)&=E\otimes_CC_I
}
\end{align}

We choose a $C$-linear retraction of $\iota$ in \eqref{eq:can-falext} and denote by
\begin{align}\label{eq:spesplit}
\pi:\ho_{\bb{Z}_p}(T_p(X),V_p(\Omega))\longrightarrow \ho_{\bb{Z}_p}(T_p(X),C(1))
\end{align}
the induced $C$-linear homomorphism.

We denote by $\tilde{\rho}$ the composition of
\begin{align}\label{eq:tilderho}
\xymatrix{
	H^0(X,\Omega_{X/K}^1)\ar[r]^-{\rho}&\ho_{\bb{Z}_p[G]}(T_p(X),V_p(\Omega))\ar[r]^-{\pi}&E(1)\ar[r]&E(1)/E^{G}(1).
}
\end{align}
where $\rho$ is the Fontaine's injection \eqref{eq:foninj}.

\begin{mylem}\label{lem:G-invECI}
	The canonical map
	\begin{align}\label{eq:sum-i-comp-abvar2}
	E^G\otimes_K K_I\longrightarrow (E\otimes_C C_I)^{G}
	\end{align}
	is an isomorphism.
\end{mylem}
\begin{proof}
	Since $E$ is a finite-dimensional $C$-vector space, the complete absolute value on $C$ extends to a complete absolute value on $E$ uniquely up to equivalence. We fix such an absolute value and still denote it by $|\ |_p$. Following \eqref{eq:sum-i-comp1} and \eqref{eq:sum-i-comp2}, the conclusion follows from the following descriptions
	\begin{align}\label{eq:sum-i-comp-abvar}
	E\otimes_C C_I&=\{(x_i)\in\prod\nolimits_{i\in I}E\ |\ \forall N>0,\ \exists \text{ finite }J\subseteq I, |x_i|_p<1/N,\ \forall i\notin J\},\\
	E^G\otimes_K K_I&=\{(x_i)\in\prod\nolimits_{i\in I}E^G\ |\ \forall N>0,\ \exists \text{ finite }J\subseteq I, |x_i|_p<1/N,\ \forall i\notin J\}.
	\end{align}	
\end{proof}

\begin{mylem}\label{lem:projfoninj}
	The map $\tilde{\rho}$ is injective, and its image lies in the $G$-invariants of $E(1)/E^{G}(1)$. Moreover, $\tilde{\rho}$ does not depend on the choice of $\pi$. Hence, we have a canonical $K$-linear injective homomorphism
	\begin{align}\label{eq:projfoninj}
	\tilde{\rho}:H^0(X,\Omega_{X/K}^1)\longrightarrow (E(1)/E^{G}(1))^{G}.
	\end{align}
\end{mylem}
\begin{proof}
	We take a $K$-basis $\{h_l\}$ of $E^G$. For any $\omega\in H^0(X,\Omega_{X/K}^1)$, thanks to \ref{lem:G-invECI}, we denote by $\sum h_l\otimes \alpha_l\in E^G\otimes_K K_I$ the image of $\omega$ in $\ho_{\bb{Z}_p}(T_p(X),C_I)$ via Fontaine's injection $\rho$ \eqref{eq:foninj} and \eqref{eq:abvar-falext}. Take any lifting $\beta_l\in V_p(\Omega)$ of $\alpha_l$ in the Faltings extension \eqref{eq:can-falext}. Consider the element
	\begin{align}
	\rho(\omega)-\sum h_l\otimes \beta_l\in \ho_{\bb{Z}_p}(T_p(X),V_p(\Omega))=E\otimes_C V_p(\Omega).
	\end{align}
	In fact, it lies in $E(1)$. For any $\sigma\in G$, 
	\begin{align}
	\sigma(\rho(\omega)-\sum h_l\otimes \beta_l)-(\rho(\omega)-\sum h_l\otimes \beta_l)=\sum h_l\otimes (\beta_l-\sigma(\beta_l))\in E^G(1).
	\end{align}
	Therefore, $\rho(\omega)-\sum h_l\otimes \beta_l$ is $G$-invariant modulo $E^G(1)$, i.e. it defines an element in $(E(1)/E^G(1))^G$. Moreover, this element does not depend on the choice of the lifting $\beta_l$. Indeed, suppose $\beta_l$ and $\beta_l'$ two liftings of $\alpha_l$, then $\beta_l'-\beta_l\in C(1)$ which shows that $(\rho(\omega)-\sum h_l\otimes \beta_l)-(\rho(\omega)-\sum h_l\otimes \beta_l')\in E^G(1)$. In particular, $\tilde{\rho}$ does not depend on the choice of $\pi$.
	
	Now we show the injectivity of $\tilde{\rho}$. Suppose that $\rho(\omega)-\sum h_l\otimes \beta_l=\sum h_l\otimes \gamma_l\in E^G(1)$. Then for any $\sigma\in G$,
	\begin{align}
	\sum h_l\otimes (\sigma(\beta_l+\gamma_l)-(\beta_l+\gamma_l))=0,
	\end{align}
	which implies that $\beta_l+\gamma_l\in V_p(\Omega)^G$=0 by \eqref{eq:falext-Ginv}. Hence $\rho(\omega)=0$, which forces $\omega$ to be zero since $\rho$ is injective.
\end{proof}

\begin{mythm}\label{thm:HTfil}
	There is a canonical exact sequence of $C$-$G_K$-modules
	\begin{align}\label{eq:HTfil}
	\xymatrix{
	0\ar[r]& H^1(X,\ca{O}_X)\otimes_KC(1)\ar[r]^-{\psi}&\ho_{\bb{Z}_p}(T_p(X),C(1))\ar[r]^-{\phi}& H^0(X,\Omega_{X/K}^1)\otimes_KC\ar[r]&0.
	}
	\end{align}
\end{mythm}
\begin{proof}
	We set $d=\dim X=\dim_K H^0(X,\Omega_{X/K}^1)$. Then $T_p(X)$ is a free $\bb{Z}_p$-module of rank $2d$. Lemma \ref{lem:projfoninj} implies that the weak Hodge-Tate weight $0$ of $E(1)$ has multiplicity $\geq d$. Let $X'$ be the dual abelian variety of $X$, and we set $E'=\ho_{\bb{Z}_p}(T_p(X'),C)$. Due to the fact that $E'=E(1)^*$ (by Weil pairing) and proposition \ref{prop:HTweight}, the weak Hodge-Tate weight $1$ of $E(1)$ has multiplicity $\geq d$. But $\dim_C E(1)=2d$, which forces these inequalities to be equalities. In particular, $\tilde{\rho}:H^0(X,\Omega_{X/K}^1)\to  (E(1)/E^{G}(1))^{G}$ is an isomorphism. Since $C(1)$ has only trivial extension by $C^{\oplus d}$ (proposition \ref{prop:repn-split}), we see that $C^{\oplus d}$ is a quotient representation of $E(1)$. By duality again, we see that $C(1)^{\oplus d}$ is a subrepresentation of $E(1)$, and thus the canonical injection $(E(1)/E^G(1))^G\otimes_K C\to E(1)/E^G(1)$ is an isomorphism. Therefore, we have a canonical surjection
	\begin{align}
	E(1)\longrightarrow H^0(X,\Omega_{X/K}^1)\otimes_KC.
	\end{align}
	By duality, $H^1(X,\ca{O}_X)\otimes_KC(1)=H^0(X',\Omega_{X'/K}^1)^*\otimes_KC(1)$ canonically identifies with a subrepresentation of $E(1)$. Now \eqref{eq:HTfil} follows from the avoidance of $C(1)^{\oplus d}$ and $C^{\oplus d}$. 
\end{proof}

\subsection{}\label{ssec:7.5}
Let's complete the proof of the main theorem \ref{thm:main}. We choose a retraction of $\iota$ in the Faltings extension \eqref{eq:can-falext}.
	By our construction, we have the following commutative diagram
	\begin{align}\label{diam:HTcomp}
	\xymatrix{
		\ho_{\bb{Z}_p}(T_p(X),C(1))\ar[r]^-{\phi}&H^0(X,\Omega_{X/K}^1)\otimes_K C\ar[d]^-{\rho}\\
		&\ho_{\bb{Z}_p}(T_p(X),V_p(\Omega))\ar[lu]^-{\pi}
	}
	\end{align}
	where $\phi$ is the surjection in the Hodge-Tate filtration \eqref{eq:HTfil}, $\pi$ is induced by the chosen retraction, and $\rho$ is the Fontaine's injection \eqref{eq:foninj}. Consider the following diagram
	\begin{align}\label{diam:HTcomp2}
	\xymatrix{
		\ho_{\bb{Z}_p[G]}(T_p(X),C(1))\ar[r]^-{\phi}&H^0(X,\Omega_{X/K}^1)\ar[d]^-{\rho}\ar[r]^-{\delta'}&H^1(G,H^1(X,\ca{O}_X)\otimes_KC(1))\ar[d]\\
		\ho_{\bb{Z}_p}(T_p(X),C(1))& \ho_{\bb{Z}_p}(T_p(X),V_p(\Omega))\ar[r]^-{\pi'}\ar[l]_-{\pi}&\ho_{\bb{Z}_p}(T_p(X),C_I)
	}
	\end{align}
	where $\delta'$ is the connecting map associated to \eqref{eq:HTfil}, where $-\pi'$ is the surjection in \eqref{eq:abvar-falext}, and where we identify $H^1(X,\ca{O}_X)$ with $\ho_{\bb{Z}_p[G]}(T_p(X),C)$ by \eqref{eq:HTfil} and identify $H^1(G,C(1))$ with $K_I$ by \eqref{eq:falext-conn}, which gives the right vertical arrow. Let $\{h_l\}$ be a $K$-basis of $H^1(X,\ca{O}_X)$. For any $\omega\in H^0(X,\Omega_{X/K}^1)$, we write $-\pi'(\rho(\omega))=\sum h_l\otimes \alpha_l$ by \ref{lem:G-invECI}, where $\alpha_l\in K_I$. Let $\beta_l\in V_p(\Omega)$ be the lifting of $\alpha_l$ via the chosen splitting of the Faltings extension. We see by the diagram \eqref{diam:HTcomp} that $\rho(\omega)-\sum h_l\otimes \beta_l$ is a lifting of $\omega$ via $\phi$. Thus, $\delta'(\omega)$ is represented by the following $1$-cocycle
	\begin{align}
	\sigma\mapsto \sum h_l\otimes (\beta_l-\sigma(\beta_l)),\ \forall \sigma\in G.
	\end{align}
	We notice that $\alpha_l\in K_I$ corresponds to a class in $H^1(G,C(1))$ represented by the following $1$-cocycle
	\begin{align}
	\sigma\mapsto \sigma(\beta_l)-\beta_l,\ \forall \sigma\in G.
	\end{align}
	In conclusion, the diagram (\ref{diam:HTcomp2}) is commutative. 

\subsection{}\label{ssec:7.6}
Now we can prove the corollary \ref{cor:main} to the main theorem.
	If the sequence \eqref{eq:HTfil} splits, then the $\phi$ in \eqref{diam:HTcomp2} is surjective. Hence $\delta'$ is zero map, and so is $\pi'\circ \rho$. Thus, the image of the Fontaine's injection $\rho$ lies in $\ho_{\bb{Z}_p}(T_p(X),C(1))$. We easily see that conversely if the image of the Fontaine's injection $\rho$ lies in $\ho_{\bb{Z}_p}(T_p(X),C(1))$, then the sequence \eqref{eq:HTfil} splits. Moreover, the splitting is unique by the avoidance of $C(1)^{\oplus d}$ and $C^{\oplus d}$.

\nocite{abbes2016p}
\bibliographystyle{amsalpha}
\bibliography{bibli}


\end{document}